\documentclass[11pt]{amsart}
\usepackage{amssymb,amsthm,eucal,amscd,verbatim}

\theoremstyle{plain}
\newtheorem{Thm}{Theorem}[section]
\newtheorem{Lem}[Thm]{Lemma}
\newtheorem{Cor}[Thm]{Corollary}
\newtheorem{Prop}[Thm]{Proposition}
\theoremstyle{definition}
\newtheorem{Def}[Thm]{Definition}
\theoremstyle{remark}

\DeclareMathOperator{\Ann}{Ann}
\DeclareMathOperator{\Tor}{Tor}
\DeclareMathOperator{\Ext}{Ext}
\DeclareMathOperator{\Supp}{Supp}

\DeclareMathOperator{\Hom}{Hom}
\DeclareMathOperator{\Ker}{Ker}
\DeclareMathOperator{\coker}{coker}

\DeclareMathOperator{\depth}{depth}

\DeclareMathOperator{\lh}{H}

\newcommand{\g}[1]{\mathfrak#1}

\newcommand{\hit}{ht}
\newcommand{\Gid}{Gid}

\newcommand{\Gfd}{Gfd}
\newcommand{\id}{id}

\newcommand{\fd}{fd}
\newcommand{\gr}{grade}
\newcommand{\Spec}{Spec}

\title{G--Gorenstein modules}

\author{mohsen aghajani}

\author{hossein zakeri}

\address{M. Aghajani and H. Zakeri, Faculty of mathematical sciences and
computer engineering, Tarbiat Moallem University, 599 Taleghani
Avenue, Tehran 15618 Iran}

\email{m\_aghajani@tmu.ac.ir}
\email{hossein\_zakeri@yahoo.com}



\begin{document}

\begin{abstract} Let $R$ be a commutative Noetherian ring. In this
paper, we study those finitely generated $R$--modules whose
Cousin complexes provide Gorenstein injective resolutions. We
call such a module a G--Gorenstein module. Characterizations of
G--Gorenstein modules are given and a class of such modules is
determined. It is shown that the class of G--Gorenstein modules
strictly contains the class of Gorenstein modules. Also, we
provide a Gorenstein injective resolution for a balanced big
Cohen--Macaulay $R$--module. Finally, using the notion of a
G--Gorenstein module, we obtain characterizations of Gorenstein
and regular local rings.
\end{abstract}

\maketitle

\section{\textbf{Introduction}}

All rings considered in this paper will be commutative and
Noetherian and will have non-zero identities; $R$ will always
denote such a ring. The Cousin complex is an effective tool in
commutative algebra and algebraic geometry. The commutative
algebra analogue of the Cousin complex of \S 2 of chapter IV of
Hartshorne \cite{H} was introduced by Sharp in \cite{Sha}. Then,
using the Cousin complex, he characterized Cohen--Macaulay and
Gorenstein rings and introduced the Gorenstein modules in
\cite{Shb}. Recall that a non-zero finitely generated $R$--module
$M$ is Gorenstein if the Cousin complex of $M$ with respect to
$M$--height filtration, $C(M)$, is an injective resolution. Note
that Cohen--Macaulay and Gorenstein rings were characterized in
terms of the Cousin complex. In 1967--69, Auslander and Bridger
introduced the concept of G--dimension for finitely generated
$R$--modules. Using this concept, it is proved that the modules
having G--dimension zero are Gorenstein projective. It is
well-known that G--dimension is a refinement of projective
dimension. Finally, in 1993--95, Enochs, Jenda and Torrecillas
extended the idea of Auslander and Bridger in \cite{E-Ja} and
\cite{E-J-T}, and introduced Gorenstein injective, projective and
flat modules (and dimensions), which all have been studied
extensively by their founders and by Christensen, Foxby,
Frankild, Holm and Xu in \cite{Ch}, \cite{Ch-F-F}, \cite{Ch-F-H},
\cite{Ch-H}, \cite{E-J-X}, \cite{Ha} and \cite{Hb}. \par Now we
briefly give some details of our results. In section 2, which
contains preliminaries, we recall some definitions which are
needed in this paper. In section 3, we establish the theory of
G--Gorenstein modules. A Finitely generated $R$--module is
G--Gorenstein if the Cousin complex of $M$ with respect to
$M$--height filtration, $C(M)$, provides a Gorenstein injective
resolution for $M$. Assume for a moment that $R$ admits a
dualizing complex. Then, in 3.3, we obtain a characterization of
G--Gorenstein modules. One can conclude from this result that a
G--Gorenstein module localizes. Also, in 3.6, we prove that a
G--Gorenstein module specializes. Theorem 3.8 determines a class
of G--Gorenstein modules. We describe finitely generated
Gorenstein projective modules by The Cousin complex over
Gorenstein local rings in 3.9. Theorem 3.11 shows that the class
of G--Gorenstein modules strictly contains the class of
Gorenstein modules. Let $R$ be a local ring and let $M$ be a
G--Gorenstein $R$--module of dimension $d$ which $\lh_{\g
m}^d(M)$ is of finite flat dimension; then, Proposition 3.12
shows that $R$ and $M$ are Gorenstein. Next, among other results,
we obtain several characterization of G--Gorenstein modules over
Cohen--Macaulay local rings. \par In section 4, we study the
balanced big Cohen--Macaulay (abbr. bbCM) modules via Cousin
complexes. Firstly, we prove, in 4.2, that if $M$ is a bbCM
$R$--module, then, under certain conditions, the Cousin complex
$C(\mathcal{D}(R),M)$ of $M$ with respect to dimension filtration
provides a Gorenstein injective resolution for $M$. Then we
establish characterizations of regular and Gorenstein local rings
in 4.8 and 4.9. Finally, in 4.10, we study both the structure of
$C(\mathcal{D}(R),M)$ and the injectivity of the top local
cohomology module of $M$ with respect to an ideal, whenever $M$
is a bbCM module over regular local ring.

\section{\textbf{Preliminaries}}

\vspace{0.3cm} In this section, we recall some definitions that we
will use later. The concept of the Cousin complex turns out to be
helpful in the theory of G--Gorenstein modules. Next we recall the
construction of the Cousin complex.
\begin{Def}
(i).\textbf{Filtration}. Following \cite{Sha}, a filtration of
$\Spec(R)$ is a descending sequence $\mathcal{F}=(F_i)_{i\geq 0}$
of subsets of $\Spec(R)$, so that
$$\Spec(R)\supseteq F_0\supseteq F_1\supseteq
F_2\supseteq\cdots\supseteq F_i\supseteq\cdots,$$

with the property that, for every $i\in\mathbb{N}_0$,
each member of $\partial F_i=F_i\backslash{F}_{i+1}$ is a minimal
member of $F_i$ with respect to inclusion. We say that the
filtration $\mathcal{F}$ admits an $R$--module $M$ if $\Supp_R
M\subseteq F_0$.
\par
(ii).\textbf{Cousin complex}. Let $\mathcal{F}=(F_i)_{i\geq 0}$ be a
filtration of $Spec(R)$ which admits an $R$--module $M$. An obvious
modification of the construction given in \S2 of \cite{Sha} will
produce a complex
\begin{center} \vspace{0.3cm}
$0 \to M\xrightarrow{d^{-1}} M^0\xrightarrow{d^0} M^1\to \cdots\to
M^i\xrightarrow{d^i} M^{i+1}\to\cdots$, \vspace{0.1cm}
\end{center}
denoted by $C(\mathcal{F},M)$ and called the Cousin complex for $M$
with respect to $\mathcal{F}$,
such that $M^0=\bigoplus_{{\g p}\in \partial F_0} M_{\g p}$;
\begin{center}
$M^i=\bigoplus_{{\g p} \in \partial F_i} (\coker d^{i-2})_{\g p}$
\end{center} \vspace{0.2cm}
for all $i>0$; the component, for $m\in M$ and $\g p\in \partial
F_0$, of $d^{-1}(m)$ in $M_\g p$ is $m/1$; and, for $i>0$, $x\in
M^{i-1}$ and $\g q\in \partial F_i$, the component of $d^{i-1}(x)$
in $(\coker d^{i-2})_\g q$ is $\pi(x)/1$, where $\pi: M^{i-1} \to
\coker d^{i-2}$ is the canonical epimorphism. \par If $M$ is an
$R$--module, then $\mathcal{H}(M)$ will denote the $M$-height
filtration $(K_i)_{i\geq0}$ of $Spec(R)$, which is defined by
\begin{center}
$K_i=\{\g p\in\Supp_R(M) | \hspace{0.5cm}\hit_M \g p\geq i \}$
\end{center} (for each $i\geq0$).
In this paper, we denote the Cousin complex for $M$ with respect to
$\mathcal{H}(M)$ by $C(M)$. Also, in \S4 we will use
$C(\mathcal{D}(R), M)$ for the Cousin complex of $M$ with respect to
the dimension filtration $\mathcal{D}(R)=(D_i)_{i\geq0}$ of the
spectrum of a local ring $R$, where $D_i$ is defined by
\begin{center}
$D_i=\{\g p\in \Spec(R) |\hspace{0.3cm}
dimR/{\g p} \leq {{dim R}-i} \}$
\end{center} (for all $i\geq0$).
\end{Def}
\begin{Def}
Following \cite{E-Jb}, an $R$--module $N$ is said to be Gorenstein
injective if there exists a $\Hom(\mathcal{I}nj,-)$ exact exact
sequence
\begin{center}\vspace{0.3cm}
$\cdots \to E_1\to E_0\to E^0\to E^1\to \cdots$
\end{center}\vspace{0.3cm}
of injective $R$--modules such that $N=\Ker(E^0\to E^1)$.
We say that an exact sequence
\begin{center}\vspace{0.3cm} $0\to N\to G^0\to G^1\to G^2\to \cdots $
\vspace{0.3cm}\end{center} of $R$--modules is a Gorenstein injective
resolution for $N$, if each $G^i$ is Gorenstein injective. We say
that $\Gid_R N \leq n$ if and only if, $N$ has a Gorenstein
injective resolution of length $n$. If there is no shorter
resolution, we set $\Gid_R N=n$. Dually, an $R$--module $M$ is said
to be Gorenstein flat if there exists an $\mathcal{I}nj\otimes -$
exact exact sequence
\begin{center} \vspace{0.3cm} $\cdots\to F_1\to F_0\to F^0\to F^1\to\cdots$
\end{center}
\vspace{0.3cm} of flat $R$--modules such that $M=\Ker(F^0\to F^1)$.
Similarly, one defines the Gorenstein flat dimension, $\Gfd_R M$, of
$M$. Finally, an $R$--module $M$ is said to be Gorenstein projective
if there is a $\Hom(-,\mathcal{P}roj)$ exact exact sequence
\begin{center} \vspace{0.3cm} $\cdots\to P_1\to P_1\to P^0\to
P^1\cdots$ \vspace{0.3cm} \end{center} of projective $R$--modules
such that $M=\Ker (P^0\to P^1)$.
\end{Def}
\begin{Def}
Following \cite{Shb}, Suppose $M$ is a non-zero finitely generated
$R$--module. Then $M$ is said to be a Gorenstein module if and only
if the Cousin complex for $M$, $C(M)$, provides an injective
resolution for $M$.
\end{Def}
\begin{Def}
Following \cite{Shd}, let $R$ be a local ring and let
$a_1,\ldots,a_d$ be a system of parameters (s.o.p) for $R$. A (not
necessarily finitely generated) $R$--module $M$ is said to be a big
Cohen--Macaulay $R$-module with respect to $a_1,\ldots,a_d$ if
$a_1,\ldots,a_d$ is an $M$--sequence, that is if
$M\neq(a_1,\ldots,a_d)M$ and, for each $i=1,\ldots,d$,
\vspace{0.3cm}\begin{center} $((a_1,\ldots,a_{i-1})M:
a_i)=(a_1,\ldots,a_{i-1})M$. \end{center} \vspace{0.3cm} An
$R$--module $M$ is said to be a balanced big Cohen--Macaulay
$R$--module if  $M$ is big Cohen--Macaulay with respect to every
system of parameters of $R$. If an $R$--module $M$ is a big
Cohen--Macaulay $R$--module with respect to some s.o.p. for $R$ and
$M$ is finitely generated, then it is well known that $M$ is a
balanced big Cohen--Macaulay $R$--module.
\end{Def}
\begin{Def}
Following \cite{X}, an $R$--module $M$ is said to be strongly
torsion free
 if  $\Tor^R_1(F, M)=0$ for any $R$--module $F$
of finite flat dimension. \end{Def}

\section{\textbf{G--Gorenstein modules}}

\vspace{0.3cm} We introduce the following definition.

\begin{Def}
Let $M$ be a non-zero finitely generated $R$-module. We say that $M$
is G--Gorenstein if and only if the Cousin complex for $M$, $C(M)$,
provides a Gorenstein injective resolution for $M$.
\end{Def}

Note that, any Gorenstein module is G--Gorenstein. In the course we
will see that there is a G--Gorenstein module which is not
Gorenstein.

The following lemma is needed in the proof of the next theorem.

\begin{Lem}

Let $S$ be a multiplicative closed subset of $R$. If $M$ is a
Gorenstein injective $S^{-1}R$--module, then $M$ is Gorenstein
injective over $R$.
\end{Lem}

\begin{proof}
 For a given injective $R$-module $E$, it is immediate to see that
 the functors ${\Hom_R(E ,-)}$ and ${Hom_{S^{-1}R}(S^{-1}E ,-)}$ are
 equivalent on the category of $S^{-1}R$--modules.
 Therefore, since every $S^{-1}R$--injective module is $R$--injective,
 the assertion follows immediately from the definition of a
 Gorenstein injective module.
 \end{proof}

 The following theorem provides a characterization of G--Gorenstein modules.

\begin{Thm}
Suppose that $R$ admits a dualizing complex and that $M$ is a
non-zero finitely generated $R$-module. Then the following
conditions are equivalent.
\begin{itemize}
\item[(i)] $M$ is G--Gorenstein.

\item[(ii)] $\depth_{R_\g p} M_\g p=\hit_M \g p=\Gid_{R_\g p} M_\g p=
\depth{R_\g p}$, for all $\g p\in\Supp_R M$.
\end{itemize}
\end{Thm}
\begin{proof}
Write C(M) as
\begin{center}
$ 0\to M\xrightarrow{d^{-1}} M^0\xrightarrow{d^{0}}
M^1\to\cdots\to M^n\xrightarrow{d^{n}} M^{n+1}\to\cdots$.
\end{center} \vspace{0.2cm} \par (i)$\Rightarrow$(ii). In view of
\cite[2.4]{Shb}, $M$ is Cohen--Macaulay; so that $\depth_{R_ \g p}
M_\g p=\hit_M\g p$ for all $\g p\in\Supp_R M$. Therefore, by
\cite[6.1.4]{B-Sh} and the main theorem of \cite{Shc}, $\,(M_\g
p)^t\cong{\lh}_{\g pR_\g p}^t(M_\g p)\neq 0\,$, where $\,t={\hit_M
\g p}\,$. Next, since ,for all $\g p\in\Supp_R M$, $\,[C_R(M)]_\g
p \cong C_{R_\g p}(M_{\g p})\,$ by \cite[3.5]{Sha} and $\,C_{R_\g
p}(M_{\g p})\,$ is an essential complex by \cite[5.3]{Sha}, we
have $\,\Gid_{R_\g p} M_\g p=t\,$ for all $\g p\in\Supp_R M$.
Therefore, by \cite[6.3]{Ch-F-H}, $\,\Gid_{R_\g p} M_\g p=\depth
R_\g p\,$, which completes the proof.
\par
(ii)$\Rightarrow$(i). Let $\,\g p\in\Supp_R M\,$. Then, by
hypothesis, $M$ is Cohen--Macaulay; so that, by \cite[2.4]{Shb},
$\,C(M)\,$ is exact. It remains to show that $M^n$ is Gorenstein
injective for all $n\geq 0$. We prove this by induction on $n$.
If $n=0$, then $\,\Gid_{R_\g p} M_\g p=0\,$ for all $\,\g
p\in\Supp_R M\,$ with $\,\hit_M\g p=0\,$; so that, by 3.2,
$\,M_\g p\,$ is a Gorenstein injective $R$--module for all $\,\g
p\in\Supp_R M\,$ with $\,\hit_M\g p=0\,$. Hence by
\cite[10.1.4]{E-Jb}, $M^0$ is Gorenstein injective. Now, assume
that $n>0$ and that $M^0$, $M^1$,\ldots, $M^{n-1}$ are Gorenstein
injective. We have the exact sequence \vspace{0.3cm}
\begin{center}
$ 0\to M\to M^0\to M^1\to\cdots\to M^{n-1}\to \coker d^{n-2}\to 0$.
\end{center} \vspace{0.3cm}
Let $\g p\in\Supp_R M$ with $\hit_M\g p=n$. Since $\Gid_{R_\g p}
M_\g p=n$ and the sequence
\begin{center} \vspace{0.2cm}
$ 0\to M_\g p\to (M^0)_\g p\to (M^1)_\g p\to\cdots\to (M^{n-1})_\g
p\to (\coker d^{n-2})_\g p\to 0$
\end{center} \vspace{0.3cm}
is exact, we deduce, by \cite[3.3]{Ch-F-H} and 3.2, that $(\coker
d^{n-2})_\g p$ is a Gorenstein injective $R$--module. Hence, by
\cite[6.9]{Ch-F-H}, $M^n=\bigoplus_{\hit_M\g p=n} (\coker
d^{n-2})_\g p$ is Gorenstein injective. This completes the inductive
step.
\end{proof}

\begin{Cor}
Suppose that $R$ admits a dualizing complex and that $M$ is a
non--zero finitely generated $R$--module. Then the following
conditions are equivalent.
\begin{itemize}
\item[(i)] $M$ is G--Gorenstein.
\item[(ii)] $M_\g p$ is a G--Gorenstein $R_\g p$--module for
all $\g p\in\Supp_R M$.
\item[(iii)] $M_\g m$ is a G--Gorenstein $R_\g m$--module for
all maximal $\g m\in\Supp_R M$.
\end{itemize}
\end{Cor}
\begin{proof}
The only non--obvious point is (iii)$\Rightarrow$(i). To this end,
let $\g p\in\Supp_R M$ and $\g m$ be a maximal ideal of $R$ which
contains $\g p$. Since $M_\g m$ is a G--Gorenstein $R_\g m$--module,
one can use 3.3 and the natural isomorphism $(M_\g m)_{\g p R_\g
m}\cong M_\g p$ to deduce that $M$ is G--Gorenstein.
\end{proof}

The following proposition, establishes
a property of G--Gorenstein modules.
\begin{Prop}
Suppose that $R$ admits a dualizing complex and that $M$ is a
non--zero finitely generated G--Gorenstein $R$--module. Then, for
every finitely generated $R$--module $N$ of finite injective or
projective dimension,
\begin{center}   $\Ext_R^i(\Ext_R^j(N, M), M)=0$
\end{center}
for all integers  $i,j$  with  $0\leq i<j$.
\end{Prop}
\begin{proof}
Since $M$ is G--Gorenstein, $C(M)$ provides a Gorenstein injective
resolution for $M$; and hence $M$ is Cohen--Macaulay by
\cite[2.4]{Shb}. Suppose that $j\geq0$ and that $N$ is a finitely
generated $R$--module of finite injective or projective dimension
with $E=\Ext_R^j(N, M)\neq 0$. Let $\g a=\Ann_R E$. Then by
\cite[1.2.10]{B-H}, it is sufficient to show that $\gr (\g a,
M)=\hit_M \g a\geq j$. To this end, it is enough to prove that $E_\g
p=0$ for all $\g p\in\Supp_R M$ with $\hit_M\g p<j$. Since $N$ is
finitely generated and by 3.4, $M_\g p$ is a G--Gorenstein $R_\g
p$--module with $\Gid_{R_\g p} M_\g p=\hit_M\g p<j$, it follows, in
view of \cite[2.22]{Hb} and \cite[19.1]{M}, that $E_\g p=0$, as
required.
\end{proof}
\begin{Thm}
Suppose that $R$ is a ring which admits a dualizing complex and that
$M$ is a G--Gorenstein $R$--module. Suppose, also, that
$x=(x_1,\ldots,x_n)$ is both an $M$--sequence and an $R$--sequence.
Then the $R/xR$--module $M/xM$ is G--Gorenstein.
\end{Thm}
\begin{proof}
It is sufficient to prove this when $n=1$. Put $\bar{M}=M/x_1M$ and
$\bar{R}=R/x_1R$. Let $\g p\in\Supp_R M/x_1M$ and let $\bar{\g p}=\g
p/x_1R$. Since $M$ is G--Gorenstein, we can see, in view of 3.3,
that \par \hspace{5.2cm} $\depth {\bar{R}_{\bar {\g
p}}}=\depth_{\bar{R}_{\bar {\g p}}} {\bar{M}_{\bar {\g p}}}$.
\hspace{5.2cm} \vspace{0.2cm} $(\ast)$ On the other hand, since
$\Gid_R M<\infty$, one can use the exact sequence
\begin{center}
$0\to M\xrightarrow{x_1}\ M\to M/x_1M\to 0$
\end{center}
to see, in view of \cite[2.25]{Hb}, that $\Gid_R {\bar{M}}<\infty$;
and so we have $\Gid_{\bar{R}} {\bar{M}}<\infty$ by \cite[11.69]{R}
and \cite[2.8]{Ch-H}. Thus we have $\Gid_{\bar{R}_{\bar {\g p}}}
{\bar{M}_{\bar {\g p}}}<\infty$ by \cite[5.5]{Ch-F-H}; and hence
$\Gid_{\bar{R}_{\bar {\g p}}} {\bar{M}_{\bar{\g p}}}=\depth
{\bar{R}_{\bar{\g p}}}$ by \cite[6.3]{Ch-F-H}. Therefore, since $M$
is Cohen--Macaulay, we conclude by $(\ast)$, that
\begin{center}
$\depth_{\bar{R}_{\bar{\g p}}} {\bar{M}_{\bar{\g p}}}=\hit_{\bar{M}}
{\bar{\g p}}=\Gid_{\bar{R}_{\bar{\g p}}} {\bar{M}_{\bar{\g
p}}}=\depth {\bar{R}_{\bar{\g p}}}$
\end{center}
for all $\bar{\g p}\in\Supp_{\bar{R}} (\bar{M})$. Now, the assertion
follows immediately from 3.3.
\end{proof}
The following corollary is immediate by \cite[1.7]{Shb} and the
above theorem.
\begin{Cor}
Let $R$ and $M$ be as in the above theorem. If $x=(x_1,\cdots,x_n)$
is maximal with respect to the property of being both an
$M$--sequence and an $R$--sequence, then $M/xM$ is a Gorenstein
injective $R/xR$--module.
\end{Cor}
\textbf{Remark}. In the rest of the paper we will use the notion of
a maximal Cohen--Macaulay module. Let $R$ be a local ring with $dim
R=d$. A Cohen--Macaulay $R$--module $M$ is said to be maximal
Cohen--Macaulay if $dim_R M=d$. Note that if $M$ is a such module,
then $\hit_M {\g p}=\hit_R {\g p}$ for all $\g p\in\Supp_R M$.
\begin{Thm}
Let $(R,\g m)$ be a local ring of dimension $d$ which admits a
dualizing complex and let $M$ be a maximal Cohen--Macaulay
$R$--module with $\Gid_R M<\infty$. Then $M$ is G--Gorenstein and
$R$ is Cohen--Macaulay. In particular, $\lh_{\g m}^d (M)$, the
$d$-th local cohomology module of $M$ with respect to $\g m$, is
Gorenstein injective.
\end{Thm}
\begin{proof}

Write the Cousin complex $C(M)$ as  \vspace{0.3cm}
\begin{center}
$0\to M\xrightarrow{d^{-1}}\ M^0\xrightarrow{d^0}\ M^1\to\cdots\to
M^n\xrightarrow{d^n}\ M^{n+1}\cdots$
\end{center} \vspace{0.3cm}
and note that, by \cite[2.4]{Shb}, it is exact. Next we use
induction on $n$ to show that $(\coker d^{n-2})_\g p$ is Gorenstein
injective as an $R_\g p$--module for all $\g p\in\Supp_R M$ with
$\hit_M \g p=n$. The case where $n=0$ follows immediately from
\cite[5.5]{Ch-F-H}, \cite[6.3]{Ch-F-H} and the above remark. Now,
let $n>0$ and suppose that the result has been proved for smaller
values of $n$. Let $\g p\in\Supp_R M$ with $\hit_M \g p=n$. Pass to
localization and consider the exact sequence \vspace{0.2cm}
\begin{center}
$0\to M_\g p\to (M^0)_\g p\to (M^1)_\g p\to \cdots \to
(M^{n-1})_\g p\to (\coker d^{n-2})_\g p\to 0$.
\end{center}
Since, in view of \cite[5.5]{Ch-F-H}, \cite[6.3]{Ch-F-H} and the
above remark, we have \vspace{0.2cm}
\begin{center}
$\Gid_{R_\g p} M_\g p=\depth {R_\g p}\leq dim {R_\g p}=\hit \g
p=n$,
\end{center} \vspace{0.2cm}
one can use the above exact sequence in conjunction with the
inductive hypothesis and \cite[3.3]{Ch-F-H} to see that $(\coker
d^{n-2})_\g p$ is a Gorenstein injective $R_\g p$--module. This
completes the inductive step. It now follows from 3.2 and
\cite[6.9]{Ch-F-H} that $M^n=\bigoplus_{\hit_M \g p=n} (\coker
d^{n-2})_\g p$ is a Gorenstein injective $R$--module for all
$n\geq 0$; and hence $C(M)$ is a Gorenstein injective resolution.
Therefore $M$ is G--Gorenstein. Then, by 3.4, $\depth_R M= dim_R
M=\Gid_R M= \depth R$. Thus, since $M$ is maximal
Cohen--Macaulay, $R$ is Cohen--Macaulay. The last assertion
follows immediately from the first part and the main theorem of
\cite{Shc}.
\end{proof}
In the following proposition, we characterize finitely generated
Gorenstein projective modules in terms of G--Gorenstein modules,
over Gorenstein local rings.
\begin{Prop}
Let $R$ be a Gorenstein local ring and let $M$ be a non-zero
finitely generated $R$--module. Then the following conditions are
equivalent.
\begin{itemize}
\item[(i)] $M$ is G--Gorenstein.
\item[(ii)] $M$ is Gorenstein projective.
\end{itemize}
\end{Prop}
\begin{proof}
(i) $\Rightarrow$ (ii). According to 3.4, $M$ is maximal
Cohen--Macaulay, and so is Gorenstein projective by
\cite[11.5.4]{E-Jb}. (ii) $\Rightarrow$ (i) is a consequence of
\cite[11.5.4]{E-Jb} and 3.8.
\end{proof}
\begin{Lem}
Let $R$ be a Cohen--Macaulay local ring which admits a dualizing
complex. Suppose that every maximal Cohen--Macaulay module is of
finite injective dimension. Then $R$ is regular.
\end{Lem}
\begin{proof}
Let $k$ be the residue field of $R$. Since $k$ is finitely
generated, by \cite{A-B}, there exists an exact sequence (which
is called a Cohen--Macaulay approximation) \vspace{0.1cm}
\begin{center}
$0\to X\to M\to k\to 0$,
\end{center} \vspace{0.2cm}
where M is a maximal Cohen--Macaulay $R$--module and X is an
$R$--module of finite injective dimension. It therefore follows from
the hypothesis that $\id_R k<\infty$. Hence, by \cite[3.1.26]{B-H},
$R$ is regular.
\end{proof}
\textbf{Remark}. Let $R$ be a non-regular Cohen--Macaulay local ring
which admits a dualizing complex. Then, by 3.10, there exists at
least one maximal Cohen--Macaulay module of infinite injective
dimension.
\begin{Thm}
Let $R$ be a non-regular Gorenstein local ring. Then the class of
G--Gorenstein modules strictly contains the class of Gorenstein
modules.
\end{Thm}
\begin{proof}
It follows from the hypothesis in conjunction with the above remark
that there exists a maximal Cohen--Macaulay module $M$ of infinite
injective dimension. Now, $M$ is not a Gorenstein module, while, by
3.8 and \cite[10.1.13]{E-Jb}, it is a G--Gorenstein module.
\end{proof}
\begin{Prop}
Let $(R,\g m,k)$ be a local ring and let $M$ be a G--Gorenstein
$R$--module of Krull dimension $d$. If $\fd_R(\lh_{\g m}^d
(M))<\infty$, then $R$ and $M$ are Gorenstrin.
\end{Prop}
\begin{proof}
Since $\lh_{\g m}^d (M)$ is $\g m$--torsion, one can see that
$\Hom_R(k,\lh_{\g m}^d(M))\neq0$. On the other hand, $\lh_{\g
m}^d(M)$ is Gorenstein injective by the main theorem of \cite{Shc}.
Therefore, in view of the hypothesis and \cite[3.3]{Ha}, $R$ is
Gorenstein. Then $\lh_{\g m}^d(M)$ has finite injective dimension by
\cite[9.1.10]{E-Jb}; and so, is injective by \cite[10.1.2]{E-Jb}.
Therefore, by \cite[3.11]{Shb}, $M$ is Gorenstein.
\end{proof}
\begin{Def}
Let $R$ be a Cohen--Macaulay local ring of Krull dimension $d$
which admits a dualizing complex and let $\omega$ be the
dualizing module of $R$. Following \cite{E-J-X}, let
$\mathcal{I}_0 (R)$ be the class of $R$--modules $N$ which
satisfies the following conditions.
\begin{itemize}
\item[(i)] $\Ext_R^i(\omega, N)=0$ , for all $i>0$.
\item[(ii)] $\Tor^R_i(\omega, \Hom_R(\omega,N))=0$, for all $i>0$.
\item[(iii)] The natural map $\omega\otimes_R\Hom_R(\omega, N)\to N$
is an isomorphism.
\end{itemize}
This class of $R$--modules is called the Bass class.
\par \vspace{0.5cm}
In the rest of this section, we assume that $(R,\g m)$ is a
Cohen--Macaulay local ring of Krull dimension $d$ which admits a
dualizing complex.
\end{Def}
\begin{Thm}
Let $M$ be a maximal Cohen--Macaulay $R$--module. Suppose that
$x=(x_1,\ldots,x_n)$ is an $R$--sequence, then the following
conditions are equivalent.
\begin{itemize}
\item[(i)] $\Gid_R M<\infty$.
\item[(ii)] $\Gid_{R/xR} (M/xM)<\infty$.
\end{itemize}
\end{Thm}
\begin{proof}
(i)$\Rightarrow$(ii). This follows from \cite[2.25]{Hb},
\cite[11.69]{R} and \cite[2.8]{Ch-H}. (ii) $\Rightarrow$ (i). We
proceed by induction on $n$. Since the general case uses the same
argument as the case where $n=1$, we provide a proof for the case
$n=1$.
\par
To this end, set $\bar{M}={M/x_1 M}$ and $\bar{R}={R/x_1 R}$, and
let $\bar{\omega}=\omega/{x_1 \omega}$, where $\omega$ is the
dualizing module of $R$. In order to prove the assertion, it is
enough, by \cite[10.4.23]{E-Jb}, to show that $M\in \mathcal{J}_
0(R)$. Therefore we need only to check the above three requirements.
Since, by hypothesis, $\bar{M}\in \mathcal{J}_0 (\bar{R})$, we have
by \cite[p.140,lemma 2]{M}, $\Ext_R^i(\bar{\omega}, M)=0$, for all
$i\geq 2$. Now, one can use the exact sequence \vspace{0.2cm}
\begin{center}
$\cdots\to \Ext_R^i(\omega, M)\xrightarrow{x}\Ext_R^i(\omega, M)\to
\Ext_R^{i+1}(\bar{\omega}, M)\to \cdots$
\end{center} \vspace{0.2cm}
and Nakayama's lemma to see that $\Ext_R^i(\omega, M)=0$\, for all
$i>0$; hence the requirement (i) holds. To prove the requirement
(ii), we can use \cite[3.3.3]{B-H} and \cite[p.140,lemma 2]{M} to
see that \vspace{0.5cm}
\par
\hspace{1.5cm} $\Tor^{\bar{R}}_i(\bar{\omega},
\Hom_{\bar{R}}(\bar{\omega}, \bar{M}))
\cong\Tor^{\bar{R}}_i(\bar{\omega}, \Hom_R(\omega,
M)\otimes_R{\bar{R}})$
\par \vspace{0.2cm}
\hspace{5.35cm}  $\cong \Tor^R_i(\bar{\omega}, \Hom_R(\omega, M)$,
for all $i\geq0$. \vspace{0.3cm} \par Therefore,
$\Tor^R_i(\bar{\omega}, \Hom_R(\omega, M))=0$, for all $i>0$. Now,
using the same argument as above, we deduce $\Tor^R_i(\omega,
\Hom_R(\omega, M))=0$, for all $i>0$. It remains only the proof of
the requirement (iii). To this end, by hypothesis, we have
\par
\vspace{0.5cm} \hspace{3cm}$\bar{R}\otimes_R M\cong \bar{M}\cong
\bar{\omega}\otimes_{\bar{R}} \Hom_{\bar{R}}(\bar{\omega}, \bar{M})$
\vspace{0.2cm}
\par
\hspace{5.45cm}$\cong \bar{\omega}\otimes_{\bar{R}}(\bar{R}\otimes_R
\Hom_R(\omega, M))$ \vspace{0.2cm}
\par \hspace{5.3cm} $\cong \bar{\omega}\otimes_R \Hom_R(\omega, M)$
\vspace{0.2cm}
\par \hspace{5.43cm}
$\cong \bar{R} \otimes_R (\omega \otimes_R \Hom_R(\omega, M))$
\vspace{0.5cm}
\par
Hence, by \cite[3.3.2]{B-H}, $M\cong \omega \otimes_R\Hom_R(\omega,
M)$. It therefore follows that $M\in \mathcal{J}_0 (R)$.
\end{proof}
\begin{Thm}
Let  $M$ be a non-zero finitely generated $R$--module. Then the
following conditions are equivalent.
\begin{itemize}
\item[(i)] $M$ is G--Gorenstein.
\item[(ii)] $\depth_R M=dim_R M=\Gid_R M=\depth R$.
\item[(iii)] $M$ is a balanced big Cohen--Macaulay module with
$\Gid_R M<\infty$.
\item[(iv)] For any sequence $x=(x_1,\ldots,x_n)$ which is maximal
with respect to the property of being both an $M$--sequence and an
$R$--sequence, $M/xM$ is a Gorenstein injective $R/xR$--module.
\item[(v)] For some sequence $x=(x_1,\ldots,x_n)$ which is maximal
with respect to the property of being both an $M$--sequence and an
$R$--sequence, $M/xM$ is a Gorenstein injective $R/xR$--module.
\end{itemize}
\end{Thm}
\begin{proof}
(i) $\Leftrightarrow$ (ii). This follows from 3.4. (ii)
$\Rightarrow$ (iii). This is clear, since $M$ is a maximal
Cohen--Macaulay module. (iii) $\Rightarrow$ (iv). It follows by the
hypothesis that $M$ is maximal Cohen--Macaulay; and so $M$ is a
G--Gorenstein $R$--module by 3.8. Now the claim is immediate by 3.7.
Since (iv)$\Rightarrow$(v) is obvious, it remains to prove the
implication (v)$\Rightarrow$(ii). To this end, let
$x=(x_1,\ldots,x_n)$ be a sequence of elements of $R$ which
satisfies the hypothesis. Then, according to \cite[6.3]{Ch-F-H},
$\depth{R/xR}=\Gid_{R/xR} {M/xM}=0$. Therefore
\begin{center} $\depth_R M=dim_R M=\depth R=dim R=n$. \end{center}
Hence, $M$ is maximal Cohen--Macaulay; and so, by 3.14, it has
finite Gorenstein injective dimension. Therefore by
\cite[6.3]{Ch-F-H}, $\Gid_R M=\depth R$. This completes the proof.
\end{proof}
\begin{Prop}
Let $M$ be a G--Gorenstein $R$--module. Suppose that $N$ is a
Cohen--Macaulay $R$--module of finite injective or projective
dimension and that $dim_R N=s$. Then the following hold.
\begin{itemize}
\item[(i)] $\Ext_R^i(N, M)=0$ for all $i\neq{d-s}$,
\item[(ii)] $\Ext_R^{d-s} (N, M)$ is a Cohen--Macaulay
$R$--module of dimension $s$.
\end{itemize}
\end{Prop}
\begin{proof}
(i) It follows from \cite[1.2.10(e)]{B-H} that $\Ext_R^i(N, M)=0$
for all $i<{d-s}$. Next we use induction on $s$ to show that
$\Ext_R^i(N, M)=0$ for all $i>d-s$. If $s=0$, then the result
follows from \cite[6.3]{Ch-F-H} , \cite[19.1]{M} and
\cite[6.2.11]{Ch}. Suppose that $s>0$ and that $x\in \g m$ is a
non-zero divisor on $N$. Consider the exact sequence
\vspace{0.4cm}\begin{center} $\cdots \to
\Ext_R^i(N,M)\xrightarrow{x} \Ext_R^i(N, M)\to \Ext_R^{i+1}(N/xN,
M) \to \cdots$ \end{center} \vspace{0.4cm} and use induction
together with Nakayama's lemma to complete the proof.
\par (ii) We prove this by induction on $s$.
There is nothing to prove in the case where $s=0$. Suppose that
$s>0$ and that $x\in\g m$ is a non-zero divisor on $N$. Then, by
(i), we have the exact sequence \vspace{0.4cm}
\begin{center}
$0\to \Ext_R^{d-s}(N, M)\xrightarrow{x} \Ext_R^{d-s}(N, M)\to
\Ext_R^{d-s+1}(N/xN, M)\to0$.
\end{center} \vspace{0.4cm}
Now, $x$ is a non-zero divisor on $\Ext_R^{d-s}(N, M)$, and so the
assertion follows from the induction hypothesis.
\end{proof}
\begin{Prop}
Suppose that $N$ is a Gorenstein $R$--module and that $M$ is a
maximal Cohen--Macaulay $R$--module with $\Gfd_R M<\infty$. Then
$\Hom_R(M, N)$ is G--Gorenstein.
\end{Prop}
\begin{proof}
Since $\id_R N<\infty$ and $\Gfd_R M<\infty$, it follows from
\cite[2.8(c)]{Ch-H} that $\Hom_R(M, N)$ has finite Gorenstein
injective dimension. Therefore, since, by \cite[3.3.3]{B-H},
$\Hom_R(M, N)$ is maximal Cohen--Macaulay, the assertion follows
from 3.8.
\end{proof}
\begin{Prop}
Suppose that $N$ is a G--Gorenstein $R$--module and that $M$ is a
maximal Cohen--Macaulay $R$--module such that the injective
dimension of $M$, $\id_R M$, is finite. Then $\Hom_R(M, N)$ is
strongly torsion free.
\end{Prop}
\begin{proof}
By 3.16, $\Hom_R(M,N)$ is maximal Cohen--Macaulay. Now, since $\id_R
M<\infty$ and $\Gid_R N<\infty$, it follows from \cite[3.5(c)]{Ch-H}
that $\Hom_R(M,N)$ is of finite Gorenstein flat dimension.
Therefore, the assertion follows immediately from
\cite[2.8]{Ch-F-F}.
\end{proof}
\begin{Thm}
Let $M$ be a finitely generated Gorenstein projective $R$--module of
finite Gorenstein injective dimension. Then the following hold.
\begin{itemize}
\item[(i)] $M$ is G--Gorenstein.
\item[(ii)] $\lh_{\g m}^d (M)$ is Gorenstein injective.
\end{itemize}
\end{Thm}
\begin{proof}
(i) By \cite[10.2.7]{E-Jb}, $M$ is a maximal Cohen--Macaulay
$R$--module. Hence, as $\Gid_R M<\infty$, $M$ is G--Gorenstein by
3.8. \par (ii) By the main theorem of \cite{Shc}, we have $\lh_{\g
m}^d (M)\cong M^d$, where $M^d$ is the $d$--th term of $C(M)$; and
so the assertion is an immediate consequence of (i).
\end{proof}
\vspace{0.3cm} \par Notice that the assertion (ii) of the above
theorem recovers the result \cite[2.7]{S} which is proved under the
condition that $R$ is Gorenstein.
\par
\textbf{Remark}. Note that if $R$ is a G--Gorenstein $R$--module,
then, by \cite[2.1]{Ha}, one can see that $R$ is a Gorenstein ring.
Therefore we are not going to define the G--Gorenstein ring.
\vspace{0.3cm}

\section{\textbf{Balanced big Cohen--Macaulay modules}}

\vspace{0.5cm} In the proof of the next lemma, we use the notion
finitistic injective dimension of $R$, denoted by FID($R$), which
is defined as
\par
\vspace{0.5cm} FID($R$)=sup\big\{${\id_RM \big{|} M}$ is an
$R$--module of finite injective dimension\big\}.
\par
\vspace{0.5cm}
\begin{Lem}
Suppose that $R$ admits a dualizing complex and that $M$ is an
$R$--module. Then the following conditions are equivalent.
\begin{itemize}
\item[(i)] $\Gid_R M < \infty$.
\item[(ii)] $\Gid_R M \leq dimR$.
\end{itemize}
\end{Lem}
\begin{proof}
(i)$\Rightarrow$ (ii). We have $\Gid_R M \leq$ FID($R$) by
\cite[3.3]{Ch-F-H}. Therefore the assertion follows from
\cite[5.5]{B} and \cite[II. Theorem 3.2.6]{R-G}. (ii)
$\Rightarrow$ (i). Since $R$ admits a dualizing complex, we see
by \cite[V.7.2]{H} that $dim R$ is finite; so that $M$ has finite
Gorenstein injective dimension.
\end{proof}

\begin{Thm}
Suppose that $R$ is a local ring of Krull dimension $d$, which
admits a dualizing complex and that $M$ is a balanced big
Cohen--Macaulay $R$--module with $\Gid_R M<\infty$. Then
$C(\mathcal{D}(R), M)$ provides a Gorenstein injective resolution
for $M$.
\end{Thm}
\begin{proof}
Write $C(\mathcal{D}(R), M)$ as  \vspace{0.3cm}
\begin{center}
$0\to M\xrightarrow{d^{-1}}\ M^0\xrightarrow{d^0}\ M^1\to \cdots\to\
M^n\xrightarrow{d^n}\to M^{n+1}\to\cdots$,
\end{center} \vspace{0.3cm}
where $M^n=\bigoplus_{dimR/\g p=d-n} (\coker d^{n-2})_\g p$. Since
$M$ is balanced big Cohen--Macaulay, $C(\mathcal{D}(R), M)$ is exact
by \cite[4.1]{Shd}. Therefore it is enough to prove that $M^n$ is
Gorenstein injective for all $n\geq 0$. To this end, we proceed by
induction on $n$. If $n=0$, then we have $M^0=\bigoplus_{dim{R/{\g
p}}=d} {M_{\g p}}$. Thus, for each $\g p\in\Spec(R)$ with $dim R/\g
p=0$, we have by \cite[5.5]{Ch-F-H}, that $\Gid_{R_{\g p}} M_{\g
p}\leq \Gid_R M<\infty$; and so, by 4.1, $\Gid_{R_{\g p}} M_{\g
p}\leq \dim R_{\g p}=0$. Therefore, according to 3.2, $M_{\g p}$ is
a Gorenstein injective $R$--module. Hence, in view of
\cite[10.1.4]{E-Jb}, we see that $M^0$ is Gorenstein injective. Now,
let $n>0$ and suppose that the result has been proved for smaller
values of $n$. Let $\g p\in\Spec(R)$ with $dimR/{\g p}=d-n$. We
obtain the exact sequence
\par
\vspace{0.4cm} \hspace{2.5cm} $0\to M_\g p\to (M^0)_\g p\to\cdots\to
(M^{n-1})_\g p\to (\coker d^{n-2})_\g p\to0$
\hspace{2.6cm}\vspace{0.4cm} $(\ast)$  from $C(\mathcal{D}(R), M)$.
Since $dimR_\g p\leq n$, we have $\Gid_{R_\g p} M_\g p\leq n$ by 4.1
and \cite[5.5]{Ch-F-H}. Therefore, using $(\ast)$ in conjunction
with \cite[3.3]{Ch-F-H} and the inductive hypothesis, we see that
$(\coker d^{n-2})_\g p$ is a Gorenstein injective $R_\g p$--module.
Hence $M^n$ is a Gorenstein injective $R$--module by 3.2 and
\cite[6.9]{Ch-F-H}.
\end{proof}
\begin{Cor}
Let $R$ and $M$ be as in the above theorem. Then $\lh_\g m^d (M)$
is a Gorenstein injective $R$--module.
\end{Cor}
\begin{proof}
By 4.2, $C(\mathcal{D}(R), M)$ is a Gorenstein injective resolution
for $M$. Hence $\lh_\g m^d(M)$ is Gorenstein injective by
\cite[1.8]{Shd}.
\end{proof}
\begin{Cor}
Let $R$ be a Gorenstein local ring of Krull dimension $d$ and let
$M$ be a balanced big Cohen--Macaulay $R$--module. Then
$C(\mathcal{D}(R), M)$ provides a Gorenstein injective resolution
for $M$ and hence $\lh_\g m^d (M)$ is Gorenstein injective.
\end{Cor}
\begin{proof}
The assertion is an immediate consequence of 4.2, 4.3 and
\cite[10.1.13]{E-Jb}.
\end{proof}
\textbf{Remark}. Let $R$ be a non Gorenstein Cohen--Macaulay local
ring which admits a dualizing complex. Then $R$ is a balanced big
Cohen--Macaulay $R$--module; but $R$ is of infinite Gorenstein
injective dimension by \cite[2.1]{Ha}. \par The following lemma is
assistant in the proof of 4.7, 4.8 and 4.9.
\begin{Lem}
Let $M$ be an $R$--module. Suppose that $C(\mathcal{F}, M)$ is exact
at $M,M^0,M^1,\ldots,M^t$. If $\id_R M<\infty\hspace{0.2cm} (resp.
\fd_R M<\infty)$, then we have $id_R M^i<\infty\hspace{0.2cm} (resp.
\fd_R M^i<\infty)$ for all $i=0,\ldots,t$.
\end{Lem}
\begin{proof}
We prove the injective case. The proof of the flat case is
similar. Write $C(\mathcal{F}, M)$ as
\begin{center}
$0\to M\xrightarrow{d^{-1}} M^0\xrightarrow{d^0} M^1\to\cdots\to
M^n\xrightarrow{d^n} M^{n+1}\to\cdots$,
\end{center} \vspace{0.2cm}
where $M^n=\bigoplus_{\g p\in\partial F_n} (\coker d^{n-2})_\g p$.
\par Let $\g p\in\partial F_0$.
Then we have $\id_R M_\g p\leq \id_{R_\g p} M_\g p\leq \id_R
M<\infty$. Since $R$ is Noetherian, we have \par
\vspace{0.3cm}\hspace{2cm}$\id_R M^0=\id_R (\bigoplus_{\g p\in
\partial F_0} M_\g p)\leq sup_{\g p\in
\partial F_0} \{\id_R M_\g p\}\leq \id_R M<\infty.$
\hspace{2cm} \vspace{0.3cm}$(\ast)$  Now we can obtain the short
exact sequences \par \vspace{0.5cm} $E_1 :\hspace{1.3cm} 0
\longrightarrow M\longrightarrow M^0\longrightarrow \coker
d^{-1}\longrightarrow 0$
\par\vspace{0.2cm}
$E_2 : \hspace{1.3cm}0\to \coker d^{-1}\to M^1\longrightarrow \coker
d^0\longrightarrow 0$
\par \hspace{3.5cm} \vdots \hspace{3.3cm}\vdots
\hspace{3.35cm}
\par $E_{t-1} : \hspace{0.95cm} 0\to \coker d^{t-4}\to M^{t-2}\to
\coker d^{t-3}\to 0$
\par \vspace{0.2cm}
$E_t :\hspace{1.3cm}0\to \coker d^{t-3}\to M^{t-1}\to \coker
d^{t-2}\to 0$  \vspace{0.5cm} \par \hspace{-0.4cm}from
$C(\mathcal{F}, M)$. Therefore, $\id_R (\coker d^{-1})<\infty$ by
$(\ast)$ and $E_1$. Now let $\g p\in\partial F_1$. Thus we have
$\id_R (\coker d^{-1})_\g p \leq \id_R (\coker d^{-1})<\infty$; and
consequently, \vspace{0.3cm}\par $\id_R M^1=\id_R(\bigoplus_{\g
p\in\partial F_1} (\coker d^{-1})_\g p)\leq sup_{\g p\in\partial
F_1}\{\id_R(\coker d^{-1})_\g p\}$ \vspace{0.3cm}
\par \hspace{6cm}$\leq \id_R(\coker d^{-1}) <\infty$ \vspace{0.2cm}
\par
Now, using the exact sequences $E_2,\cdots,E_t$ and employing the
same argument as above, one can prove the assertion by induction.
\end{proof}
\begin{Prop}
Let $M$ be a G--Gorenstein $R$--module. Then $M$ is Gorenstein
whenever $\id_R M<\infty$. In particular, if $R$ is a regular
local ring, then $M$ is free.
\end{Prop}
\begin{proof}
The first part of the assertion is clear by 4.5 and
\cite[10.1.2]{E-Jb}. The last part of the assertion follows
immediately from the first part in conjunction with 3.9 and
\cite[10.2.3]{E-Jb}.
\end{proof}
The next theorem provides a characterization for Gorenstein local
rings, in terms of G--Gorenstein modules.
\begin{Thm}
Let $(R,\g m, k)$ be a Cohen--Macaulay local ring of Krull dimension
$d$. Then the following conditions are equivalent.
\begin{itemize}
\item[(i)] every maximal Cohen--Macaulay $R$--module is G--Gorenstein.
\item[(ii)] every $\g m$--torsion $R$--module is of finite Gorenstein
injective dimension.
\item[(iii)] $\Gid_R (\lh_m^d (R))<\infty$.
\item[(iv)] $R$ is Gorenstein.
\end{itemize}
\end{Thm}
\begin{proof}
(i)$\Rightarrow$ (iv). Since $R$ itself is a maximal Cohen--Macaulay
$R$--module, we have $\Gid_R R<\infty$. Therefore, $R$ is Gorenstein
by \cite[2.1]{Ha}. (iv) $\Rightarrow$ (i). This is immediate by
\cite[10.1.13]{E-Jb} and 3.8. (ii) $\Rightarrow$ (iii) is clear by
the fact that $\lh_\g m^d (R)$ is $\g m$--torsion. (iii)
$\Rightarrow$ (iv). Since $R$ is Cohen--Macaulay, the complex $C(R)$
is exact by \cite[4.7]{Sha}. Hence, in view of the main theorem of
\cite{Shc} and 4.5, we have $\fd_R (\lh_\g m^d (R))<\infty$. On the
other hand, we see that $\Hom_R(k, \lh_\g m^d (R))\neq0$. Therefore,
the result follows from \cite[3.3]{Ha}. (iv) $\Rightarrow$ (ii) is
clear by \cite[10.1.13]{E-Jb}.
\end{proof}
The following theorem provides a characterization for regular
local rings.
\begin{Thm}
Let $R$ be a Gorenstein local ring. Then the following conditions
are equivalent.
\begin{itemize}
\item[(i)] every Gorenstein flat $R$--module is flat.
\item[(ii)] every balanced big Cohen--Macaulay $R$--module is
of finite flat dimension.
\item[(iii)] every G--Gorenstein $R$--module is Gorenstein.
\item[(iv)] $R$ is regular.
\end{itemize}
\end{Thm}
\begin{proof}
(i) $\Rightarrow$(ii). Let $M$ be a balanced big Cohen--Macaulay
$R$--module. Then, by \cite[10.3.13]{E-Jb}, $M$ has finite
Gorenstein flat dimension. Therefore, in view of the hypothesis,
$\fd_R M<\infty$. (ii) $\Rightarrow$ (iii). Let M be a
G-Gorenstein $R$--module. Then $M$ is a balanced big
Cohen--Macaulay $R$--module by 3.14; so that $\fd_R M<\infty$.
Hence, by \cite[9.1.10]{E-Jb}, $\id_R M<\infty$. Therefore the
terms of $C(M)$ have finite injective dimension by 4.5; and hence
they are injective by \cite[10.1.2]{E-Jb}. Thus $C(M)$ is an
injective resolution for $M$; and hence $M$ is Gorenstein. (iii)
$\Rightarrow$ (iv). Let $M$ be a maximal Cohen--Macaulay
$R$--module. As $R$ is Gorenstein, $M$ is G--Gorenstein by 3.8.
Thus, in view of (iii), $M$ is Gorenstein; so that $\id_R
M<\infty$. Therefore, since $M$ is an arbitrary maximal
Cohen--Macaulay $R$--module, the claim follows from 3.10. (iv)
$\Rightarrow$ (i). Assume that $M$ is a Gorenstein flat
$R$--module. Since $R$ is a regular local ring, it has finite
global dimension. Hence $fd_R M<\infty$. Then, by
\cite[10.3.4]{E-Jb}, $M$ is flat.
\end{proof}
\begin{Thm}
Let $(R,\g m)$ be a regular local ring of dimension $d$ and let $M$
be a balanced big Cohen--Macaulay $R$-module. Then
\begin{itemize}
\item[(i)] $C(\mathcal{D}(R), M)$ is an injective resolution for $M$
and $\lh_\g m^d (M)$ is injective.
\item[(ii)] If \hspace{0.1cm} $d\leq 2$ and $\g a$ is a non-zero
ideal of $R$, then $\lh_\g a^d (M)$ is injective.
\end{itemize}
\end{Thm}
\begin{proof}
(i) Since $R$ is Gorenstein, $C(\mathcal{D}(R), M)$ is a
Gorenstein injective resolution for $M$ and $\lh_\g m^d (M)$ is a
Gorenstein injective module by 4.4. Hence, the assertion follows
from 4.5, \cite[10.1.2]{E-Jb} and the fact that $\id_R M<\infty$
by regularity of $R$. \par (ii) Let $P$ be a projective
$R$--module. Then $\lh_\g a^d (P)$ is Gorenstein injective in
view of \cite[3.4.10]{B-Sh} and \cite[2.6]{S}; so that, since $R$
is regular, it is an injective $R$--module by \cite[10.1.2]{E-Jb}.
Now let $M$ be a balanced big Cohen--Macaulay $R$--module. Then
$M$ is flat by \cite[9.1.8]{B-H}; and hence it is the direct
limit of a family of projective $R$--modules. Therefore the
assertion follows from \cite[3.4.10]{B-Sh} and the fact that $R$
is Noetherian.
\end{proof}

\end{document}